\documentclass{scrarticle}
\usepackage{graphicx} 

\usepackage[utf8]{inputenc}
\usepackage{amsmath,amssymb,amsthm,amsfonts}
\usepackage{mathtools,thmtools}
\usepackage[british]{babel}
\usepackage{csquotes}
\usepackage{hyperref}
\usepackage[nameinlink,noabbrev]{cleveref}
\usepackage{enumerate}
\usepackage{xcolor}

\usepackage[maxbibnames=5,sorting=nyt,sortcites,giveninits=true]{biblatex}
\DeclareNameAlias{sortname}{family-given}
\newcommand{\R}     {\mathbb{R}}

\newcommand{\N}     {\mathbb{N}}
\newcommand{\Z}     {\mathbb{Z}}
\newcommand{\T}     {\mathfrak{T}}
\newcommand{\J}     {\mathfrak{J}}

\let\S\relax
\newcommand{\S}     {\mathcal{S}}

\newcommand{\V}     {\mathcal{V}}

\newcommand{\diff}  {\mathop{}\!\mathrm{d}}

\newcommand{\rel}   {\mathrm{rel}}

\renewcommand{\c}   {\mathrm{c}}

\newcommand{\E}     {\mathcal{E}}
\renewcommand{\L}   {\mathfrak{L}}
\newcommand{\hatL}  {\hat\L}

\newcommand{\hatLgamma}{\hatL^{(\gamma)}}

\let\phi\varphi

\DeclareMathOperator{\gap}  {gap}

\DeclareMathOperator{\unif} {Unif}
\DeclareMathOperator{\dom}  {Dom}

\DeclareMathOperator{\grad} {grad}

\DeclarePairedDelimiterX{\norm}[1]{\lVert}{\rVert}{#1}

\theoremstyle{plain}
\newtheorem{theorem}{Theorem}
\newtheorem{lemma}[theorem]{Lemma}
\newtheorem{corollary}[theorem]{Corollary}

\theoremstyle{definition}

\newtheorem{remark}[theorem]{Remark}
\newtheorem{remark*}{Remark}

\crefname{lemma}{lemma}{lemmas}
\crefname{theorem}{theorem}{theorems}
\crefname{assumption}{assumption}{assumptions}

\crefname{assumptionalph}{assumption}{Assumptions}

\title{Convergence rates of self-repellent random walks, their local time and Event Chain Monte Carlo}

\author{Andreas Eberle\thanks{E-Mail: \href{mailto:eberle@uni-bonn.de}{eberle@uni-bonn.de}, ORCID: \href{https://orcid.org/0000-0003-0346-3820}{0000-0003-0346-3820}}\qquad Francis Lörler\thanks{E-Mail: \href{mailto:loerler@uni-bonn.de}{loerler@uni-bonn.de}, ORCID: \href{https://orcid.org/0009-0007-3177-1093}{0009-0007-3177-1093}}\medskip\\Institute for Applied Mathematics, University of Bonn.}

\begin{document}

\maketitle

\begin{abstract}
    We study the rate of convergence to equilibrium of the self-repellent random walk and its local time process on the discrete circle $\Z_n$. While the self-repellent random walk alone is non-Markovian since the jump rates depend on its history via its local time, jointly considering the evolution of the local time profile and the position yields a piecewise deterministic, non-reversible Markov process. We show that this joint process can be interpreted as a second-order lift of a reversible diffusion process, the discrete stochastic heat equation with Gaussian invariant measure. In particular, we obtain a lower bound on the relaxation time of order $\Omega(n^{3/2})$. Using a flow Poincaré inequality, we prove an upper bound for a slightly modified dynamics of order $O(n^2)$, matching recent conjectures in the physics literature. 
    Furthermore, since the self-repellent random walk and its local time process coincide with the Event Chain Monte Carlo algorithm for the harmonic chain, a non-reversible MCMC method, we demonstrate that the relaxation time bound confirms the recent empirical observation that Event Chain Monte Carlo algorithms can outperform traditional MCMC methods such as Hamiltonian Monte Carlo.

    \begin{samepage}
    \par\vspace\baselineskip
    \end{samepage}
\end{abstract}

\section{Introduction}

In this paper, we study the rate of convergence to stationarity for the joint process of a
self-repellent random walk on a discrete circle and its local time process. This is a stronger statement than convergence to stationarity for the self-repellent random walk on its own.
The key new ingredient is the identification of the joint process of local time and the self-repellent random walk as a second order lift of a discrete stochastic heat equation in the sense of \cite{EberleLoerler2024Lifts}. This 
will allow us to apply a recently developed approach to quantitative bounds for convergence to
equilibrium of non-reversible lifts of Markov processes to self-repellent random walks.

\subsection{Self-repellent random walks}

We consider the self-repellent random walk as a continuous-time stochastic process $X(t)$ with state space given by the discrete circle $\Z_n = \Z/n\Z$, and defined on some probability space $(\Omega ,\mathcal A,\mathbb P)$. For $t\in[0,\infty)$ let $L(t)=(L_i(t))_{i\in\mathbb Z_n}$, where
\begin{equation*}
    L_i(t) = \int_0^t1_{\{i\}}(X(s))\diff s
\end{equation*}
is the time the process has spent at position $i$ up to time $t$. The self-repellent random walk performs nearest neighbour jumps
\begin{align*}
    i \to i+1&\quad\text{with rate}\quad q^+(l,i) = (l_i')_-\,,\quad \text{and}\\
    i \to i-1&\quad\text{with rate}\quad q^-(l,i) = (l_{i-1}')_+\,,
\end{align*}
where $l$ is the current value of the local time process $L(t)$ and $$l_i' = l_{i+1}-l_i\, .$$ 
Here and in the following, indices are to be understood periodically in $\Z_n$. 
The process $X(t)$ on its own has memory and is hence not a Markov process, but
the process $(L(t)),X(t))_{t\geq 0}$
is a piecewise-deterministic Markov process with state space $\R^n\times\Z_n$. The jump rates $q^+$ and $q^-$ are chosen such that the walker prefers to jump towards neighbouring states at which the random walk has spent less time. Such self-repellent random walks were first introduced by Amit, Parisi and Peliti \cite{AmitParisiPeliti1983Asymptotic} in discrete time as a model for polymer growth. The continuous-time dynamics was introduced in \cite[Section 1.2]{TothVeto2011Continuous} on the infinite grid $\R^\Z\times\Z$ and termed the continuous-time `true' self-avoiding random walk. 
Note that for any $t\in [0,\infty )$,
\begin{equation}\label{eq:growth}
    \sum_{i=1}^n L_i(t)=t\, .
\end{equation}
In order to obtain a process that converges to stationarity, we project the local time $l$ onto the subspace  $$\S  = \big\{\ell\in\R^n : \sum_{i=1}^n\ell_i = 0\big\}$$ of mean-zero vectors by subtracting the average, i.e.\ we consider the process $\mathcal L (t)$ with state space $\mathcal S$ and components
\begin{equation}\label{eq:di_projected}
    \mathcal L_i(t) = L_i(t)-\frac 1n\sum_{j=1}^nL_j(t) = L_i(t)-\frac{t}{n}\qquad\text{for }i\in\mathbb Z_n\,.
\end{equation}
Note that functions $f\colon\mathcal S\to\mathbb R$ can be extended to $\mathbb R^n$ by 
setting $\hat f(l)=f(l-  l\cdot\mathfrak n \, \mathfrak n )$, where $\mathfrak{n}=n^{-1/2}\mathbf{1}$ with $\mathbf{1}=(1,\dots,1)^\top\in\R^n$ is the unit normal vector to $\mathcal S$. In this sense, we define 
\begin{eqnarray}\label{projectedvectorfields}
  \partial_{\ell_i}f(\ell )\ =\ \partial_{\ell_i}\hat f(\ell )\ = \ \partial_{e_i-n^{-1}\mathbf{1}}f(\ell)\,.  
\end{eqnarray}
The corresponding vector fields $\partial_{\ell_i}$ are the projections of the canonical basis vector fields $e_i\in\mathbb R^n$ to $\S$. Since the jump rates $q^+$ and $q^-$ only depend on the increments of $l$ which coincide with those of $\ell$, the process $(\mathcal L(t),X(t))_{t\geq 0}$
is also a piecewise-deterministic Markov process with state space $\S\times\Z_n$. 
On functions $f\in C^1_b(\S\times\Z_n )$, its generator is
\begin{eqnarray*}
     \hatL f (\ell,i)& =& \T f (\ell,i)+ \J_+f(\ell,i)+ \J_-f(\ell,i),\quad\text{where}\\ 
      \T f(\ell,i) & =&  \partial_{\ell_i}f(\ell,i)\,,\\
    \J_\pm f(\ell,i) & =&  q^\pm (\ell,i)\bigl(f(\ell,i\pm 1)-f(\ell, i)\bigr)\,.
\end{eqnarray*}

Next we identify the stationary distribution. Consider the Laplacian on $\Z_n$ given by
\begin{eqnarray}\label{eq:Deltan}
    \Delta_n &=& \begin{pmatrix}
        -2&1&0&\cdots&1\\
        1&-2&1&\cdots&0\\
        0&1&-2&\cdots&0\\
        \vdots&\vdots&\vdots&\ddots&\vdots\\
        1&0&0&\cdots&-2
    \end{pmatrix}\in\R^{n\times n}\,,\ \ 
\end{eqnarray}
that is $ (\Delta_nl)_i = l_{i+1}-2l_i+l_{i-1} = l_i^\prime-l_{i-1}^\prime\,$.
Since $\ker\Delta_n = \langle\mathbf{1}\rangle\subset\R^n$, $\Delta_n\colon \S\to\S$ is a bijection.
Hence the probability measure
\begin{equation*}
    \mu = \mathcal{N}(0,-\Delta_n^{-1})
\end{equation*}
is well-defined and {supported on }$\S$.
Note that $\mu$ has a density proportional to $\exp(-U)$ with respect to the $(n-1)$-dimensional Haar measure on $\S$, where
\begin{equation*}
    U(\ell) = -\frac{1}{2}\ell^\top\Delta_n\ell= \frac{1}{2}\sum_{i=1}^n(\ell_{i+1}-\ell_{i})^2 .
\end{equation*}

\begin{lemma}\label{lem:LhatProperties} 
    \begin{enumerate}[(i)]
        \item The probability measure $\hat\mu = \mu\otimes\unif(\Z_n)$ is invariant for the Markov process $(\mathcal L(t),X(t))_{t\geq 0}$.
        \item Functions $f\in C_b^1(\S\times\mathbb Z_n)$ are contained in the domains of the adjoints of the operators $\T$, $\J_+$, $\J_-$ and $\hatL$ in $L^2(\hat\mu)$, and
        \begin{align*}
            \T^*f(\ell,i) &\ =\ -\T f(\ell,i) - (\Delta_n\ell)_if(\ell,i)\,,\\
            \J_\pm ^*f(\ell,i) &\ =\ q^\pm (\ell,i\mp 1)f(\ell,i\mp 1) - q^\pm (\ell,i)f(\ell,i)\,,\\
             \hatL ^*f(\ell,i)&\ =\ -\T f(\ell,i)+ (\ell_i')_+\big(f(\ell,i+1)-f(\ell,i)\big) + (\ell_{i-1}')_-\big(f(\ell,i-1)-f(\ell,i)\big)\,.
        \end{align*}
    \end{enumerate}
\end{lemma}

\begin{proof}
By an elementary computation based on integration and summation by parts, one verifies that for any $f,g\in C_b^1(\S\times\mathbb Z_n)$ and for $\mathfrak A\in\{ \T ,\J_+,\J_-\}$,
\begin{eqnarray}\label{eq:ibp}
    \sum_{i=1}^n\int f(\ell ,i)\, (\mathfrak A g)(\ell ,i)\, \mu (\diff\ell )\ =\ \sum_{i=1}^n\int (\mathfrak A^*f)(\ell ,i)\, g(\ell ,i)\, \mu (\diff\ell ), 
\end{eqnarray}
where $\T^*$, $\J_+^*$ and $\J_-^*$ are given by the expressions in the statement of the lemma. Consequently, a corresponding identity holds for $\mathfrak A=\hatL$ with
\begin{align*}
            \MoveEqLeft\hatL ^*f(\ell,i) = \T^*f(\ell,i) + \J_+^*f(\ell,i) + \J_-^*f(\ell,i)\\
            &=-\T f(\ell,i) - (\ell_i'-\ell_{i-1}')f(\ell,i) + (\ell_i')_+f(\ell,i+1) - (\ell_i')_-f(\ell,i)\\
            &\qquad + (\ell_{i-1}')_-f(\ell,i-1) - (\ell_{i-1})_+f(\ell,i)\\
            &= - \T f(\ell,i) + (\ell_i')_+\big(f(\ell,i+1)-f(\ell,i)\big) + (\ell_{i-1}')_-\big(f(\ell,i-1)-f(\ell,i)\big) .
        \end{align*}
In particular, since $\hatL^*1=0$,
$$\int \hatL f \diff\hat\mu\ =\ \frac 1n\sum_{i=1}^n\int \hatL f(\ell ,i)\, \mu (\diff\ell )
\ =\  0$$
holds for all $f\in C_b^1(\S\times\mathbb Z_n)$. This implies invariance of $\hat\mu$ by 
\cite[Theorem 21]{Durmus2021PDMP}. The second part of the claim now follows directly from \eqref{eq:ibp}.
\end{proof}

By \Cref{lem:LhatProperties}, the transition function of the Markov process
$(\mathcal L(t),X(t))_{t\ge 0}$ induces a strongly continuous contraction semigroup on $L^2(\S\times\mathbb Z_n,\hat\mu )$ that we denote by $(\hat P_t)_{t\ge 0}$. The generator of this semigroup and its domain are denoted by $(\hatL ,\dom(\hatL ))$, it extends the operator 
$(\hatL ,C_b^1(\S\times\mathbb Z_n)) $ introduced above.
  
\begin{remark}
    \Cref{lem:LhatProperties} shows that for $f\in C_b^1(\S\times\mathbb Z_n)$, the symmetric and antisymmetric part 
    of the generator $\hatL$ are given by
    \begin{eqnarray*}
        \frac{1}{2}(\hatL+\hatL^*)f(\ell,i) &=& |\ell_i'|\big(f(\ell,i+1)-f(\ell,i)\big) + |\ell_{i-1}'|\big(f(\ell,i-1)-f(\ell,i)\big)\, ,\\
        \frac{1}{2}(\hatL-\hatL^*)f(\ell,i)&=&\T f(\ell,i)-\ell_i'\big(f(\ell,i+1)-f(\ell,i)\big) + \ell_{i-1}'\big(f(\ell,i-1)-f(\ell,i)\big)\,.
    \end{eqnarray*}  
    Thus the symmetric part is the generator of a jump process on the second component with rates depending on the first component, whereas the antisymmetric part is not the generator of a Markov process.
\end{remark}

\subsection{The self-repellent random walk as a second-order lift}

Inspired by the notion of lifts of Markov chains in discrete time \cite{Diaconis2000Lift,Chen1999Lift}, second-order lifts were introduced as a counterpart in continuous time and space in \cite{EberleLoerler2024Lifts}. They provide a suitable framework for the study of convergence to equilibrium of many non-reversible dynamics by viewing them as lifts of reversible diffusions. Intuitively, a
Markov process on a product space $\mathcal S\times\mathcal V$ with invariant probability measure $\mu \otimes\kappa$ is a second-order lift of a Markov process
with state space $\mathcal S$ and invariant measure $\mu$ if in the limit as $t\downarrow 0$, the transition semigroups
$(\hat P_t)_{t\ge 0}$ and $(P_t)_{t\ge 0}$ satisfy
\begin{equation}\label{secondorderliftsemigroup}
 \int_\V \hat P_t(f\circ\pi )(\ell ,x)\,\kappa (\diff x)=(P_{t^2}f)(x)+o(t^2)    
\end{equation}
for functions $f$ in the domain of the generator $\mathfrak L$ of $(P_t)_{t\ge 0}$, where $\pi\colon\S\times\mathcal V\to\S$ is the canonical projection $\pi(\ell,x)=\ell$. A precise formal definition is given below. 

Although a Markov process and its lift can behave very differently for long times, the lift property \eqref{secondorderliftsemigroup} can be used to study convergence to equilibrium of the lifted process based on spectral properties of $\mathfrak L$ \cite{EberleLoerler2024Lifts, EberleLoerler2024Spacetime, EGHLM2025Convergence}.

We will show that the self-repellent random walk considered above is a 
second-order lift of the reversible diffusion process 
$(Z(t))_{t\geq 0}$ with state space $\S$ that solves the discrete stochastic heat equation
\begin{equation}\label{eq:SHE}
    \diff Z(t)\ =\ \frac{1}{2n}\Delta_nZ(t)\diff t \;+\; \frac{1}{\sqrt{n}}\diff W(t)\,.
\end{equation}
Here $W$ is a Brownian motion on $\S$, i.e.\ $W(t)=B(t)-n^{-1}B(t)\cdot\mathbf{1}\,\mathbf{1}  $ is the orthogonal projection of a Brownian motion $B$ in $\mathbb R^n$ to 
the linear subspace $\S$. The process $(Z(t))_{t\geq 0}$ is a reversible diffusion process with invariant measure $\mu$, and
its $L^2(\mu )$ generator $(\L,\dom(\L))$ is given by 
 \begin{equation}\label{generatorOLD}
    \L g(\ell)\ =\ \frac{1}{2n}\sum_{i=1}^n\left(\partial_{\ell_i}^2g(\ell) + (\Delta_n\ell)_i\partial_{\ell_i}g(\ell)\right)\ =\ -\frac{1}{2n}\grad^*\grad g(\ell),
\end{equation}
where $\grad$ is the gradient on $\S$, the adjoint is in $L^2(\mu)$, and $\dom(\L) = H^{2,2}(\S ,\mu)$. Note that this is an Ornstein-Uhlenbeck process slowed down by a factor $1/(2n)$. In particular, the spectral gap of the generator coincides with that of $\frac{1}{2n}\Delta_n$, see \cite{Metafune2002Spectrum}. The eigenvectors of $\Delta_n$ are $\alpha_k = (\cos(\frac{2\pi k}{n}i))_{i\in\Z_n}$, $k\in\{0,\dots,n-1\}$ with $\Delta_n\alpha_k = 2(1-\cos(\frac{2\pi k}{n}))\alpha_k$. Therefore,
\begin{equation}\label{eq:gapcollapse}
    \mathrm{gap}(\L )\ =\ \frac{1}{2n}\mathrm{gap}(\Delta_n)\ =\ \frac 1n\left( 1-\cos\left({2\pi}/n\right)\right)\ \in\ \Theta (n^{-3})\,.
\end{equation}

We now state the formal definition of a second-order lift that is motivated by a Taylor expansion of \eqref{secondorderliftsemigroup}, see \cite{EberleLoerler2024Lifts} for further motivation and examples.
The process $(\mathcal L(t),X(t))_{t\geq 0}$ with $L^2(\hat\mu )$ generator $(\hat\L ,\dom (\hat\L ))$ is called a \emph{second-order lift} of $(Z(t))_{t\geq 0}$, if there exists a core $C$ of $(\L,\dom(\L))$ such that
\begin{align}\label{eq:deflift0}
    f \circ \pi \in \dom(\hatL) \text{ for all } f \in C
\end{align}
and for all $f,g\in C$ we have
\begin{equation}\label{eq:deflift1}
     \big< \hatL (f \circ \pi), g \circ \pi \big>_{L^2(\hat\mu)} = 0\,,
\end{equation}
and
\begin{equation}\label{eq:deflift2}
    \frac{1}{2}\big< \hatL (f \circ \pi), \hatL(g \circ \pi) \big>_{L^2(\hat\mu)} = -\langle f, \L g\rangle_{L^2(\mu)}\,.
\end{equation}
Conversely, we refer to the process $(Z(t))_{t\ge 0}$ as the \emph{collapse} of $(\mathcal L(t),X(t))_{t\geq 0}$.

\begin{theorem}\label{thm:SRMlift}
    The self-repellent random walk $(\mathcal L(t),X(t))$ is a second-order lift of $(Z(t))$.
\end{theorem}

\begin{proof}
    The smooth compactly supported functions $C_\c^\infty(\S)$ form a core for $(\L,\dom(\L))$ \cite{Wielens1985Selfadjointness}. For $f\in C_\c^\infty(\S)$ we have $f\circ\pi\in\dom(\hatL)$ and $\hatL (f\circ\pi)(\ell,x) = \partial_{\ell_x}f(\ell)$, so that
    \begin{equation*}
     \big< \hatL (f \circ \pi), g \circ \pi \big>_{L^2(\hat\mu)} =  \frac{1}{n}\sum_{i=1}^n \int \partial_{\ell_i}f(\ell)\, g(\ell )\, \mu (\diff\ell )= 0\,,
    \end{equation*}
    which implies \eqref{eq:deflift1}. Here we have used that $\partial_{\ell_i}$ is the directional derivative in direction $e_i-n^{-1}\mathbf{1}$, and these vectors sum to $0$. Furthermore,
    \begin{equation*}
        \frac{1}{n}\sum_{i=1}^n \big(\hatL(g\circ\pi)(\ell,i)\big)^2 = \frac{1}{n}\sum_{i=1}^n(\partial_{\ell_i}g(\ell))^2 = \frac{1}{n}\left|\grad g(\ell)\right|^2\,,
    \end{equation*}
    so that $\frac{1}{2}\int(\hatL(g \circ\pi))^2\diff\hat\mu = -\int g\L g\diff\mu$, which implies \eqref{eq:deflift2} by polarisation.
\end{proof}

\subsection{Bounds for relaxation times}

Although the invariant measure is a product of a Gaussian and a uniform distribution, sharp bounds on the exponential rate of convergence to equilibrium are not known for the Markov process $(\mathcal L(t),X(t))$. Even the order of the dependence of the convergence rate on $n$ is unknown. The reason is that on the one hand, the process is non-reversible so that 
standard tools do not apply, and on the other hand, the combined dynamics consisting 
of jumps in the second component and deterministic moves in the first component does
not reflect the special structure of the Gaussian measure in an obvious way.

\medskip

We first note that the representation of the self-repellent random walk as a lift of the stochastic heat equation enables us to give a lower bound on the $L^2$-relaxation time
$$t_{\rel}(\hat P)\ =\ \inf\left\{ t\ge 0 :\| \hat P_tf\|_{L^2(\hat\mu )}\le e^{-1}\| f\|_{L^2(\hat\mu )}\text{ for all }f\in L_0^2(\hat\mu )\,\right\} .$$
Here $L_0^2(\hat\mu )$ denotes the orthogonal complements of the constant functions in $L^2(\hat \mu )$.
Theorem 11 in \cite{EberleLoerler2024Lifts} provides a lower bound on the relaxation time of arbitrary second-order lifts in terms of the square root of the spectral gap of the collapse. We give a slight improvement in terms of the constant prefactor with a short self-contained proof in \Cref{thm:lowerbound}.

\begin{theorem}[Lower bound on relaxation time]\label{thm:lowerbound}
Let $(P_t)_{t\geq 0}$ be the transition semigroup of a reversible Markov process and $(\hat P_t)_{t\geq 0}$ an arbitrary second-order lift. Then
\begin{equation*}
    t_\rel(\hat P)\ \geq \ \frac{1-e^{-1}}{\sqrt{2}}\sqrt{t_\rel(P)}\,.
\end{equation*}
In particular, for the self-repellent random walk,
   $$ t_{\rel}(\hat P)\ \geq \ \frac{1-e^{-1}}{2\pi} n^{3/2}  \,.$$
\end{theorem}
\begin{proof}
    For any $f\in\dom(\hatL)\cap L_0^2(\hat\mu)$ and $t\geq 0$, we have
    \begin{equation*}
        \norm{\hat P_t f-f}_{L^2(\hat\mu)} \leq \int_0^t\norm{\hat P_s\hatL f}_{L^2(\hat\mu)}\diff s \leq t\norm{\hatL f}_{L^2(\hat\mu)}\,,
    \end{equation*}
    so that
    \begin{equation*}
        \norm{\hat P_t f}_{L^2(\hat\mu)} \geq \norm{f}_{L^2(\hat\mu)} - \norm{\hat P_t f-f}_{L^2(\hat\mu)} \geq \norm{f}_{L^2(\hat\mu)} - t\norm{\hatL f}_{L^2(\hat\mu)}\,.
    \end{equation*}
    For any $g\in L_0^2(\mu)$ and $\lambda>0$ such that $\langle g,-\L g\rangle_{L^2(\mu)} = \lambda\norm{g}_{L^2(\mu)}^2$, we have $\norm{\hatL(g\circ\pi)} = \sqrt{2\lambda}\norm{g\circ\pi}_{L^2(\hat\mu)}$ by the second-order lift property \eqref{eq:deflift2}, so that
    \begin{equation*}
        \norm{\hat P_t(g\circ\pi)}_{L^2(\hat\mu)} \geq (1-\sqrt{2\lambda}t)\norm{g\circ\pi}_{L^2(\hat\mu)}\qquad\text{for all }t\geq 0\,.
    \end{equation*}
    In particular, we obtain $t_\rel(\hat P)\geq \frac{1-e^{-1}}{\sqrt{2}}\lambda^{-1/2}$. Considering the limit $\lambda\to\gap(\L)$ yields the claim since $t_\rel(P) = \gap(\L)^{-1}$ by reversibility.
    The lower bound for the self-repellent random walk follows from \eqref{eq:gapcollapse} and $1-\cos(x)\leq\frac{x^2}{2}$.
\end{proof}


 \begin{remark}
        Based on numerical and heuristic evidence \cite{Krauth2024HMCvsECMC,Massoulie2025Velocity} and a Bethe ansatz calculation \cite{Essler2025Lifted} for related models, it has been conjectured in the physics literature that 
        the relaxation time 
        is
    of order $\Theta(n^2)$. Mathematically rigorous proofs for quantitative  upper and lower bounds on the relaxation time of this order are missing. Physics inspired arguments mostly concern different discrete models such as the lifted totally antisymmetric simple exclusion process \cite{EsslerKrauth2024LiftedTASEP}. Here the special combinatorial structure may allow for other approaches for bounding the relaxation time. Even if this were possible, the development of a more generally applicable analytic approach is important. 
\end{remark}

Next, as a first step in the direction of a rigorous proof of upper bounds on relaxation times, we 
consider a modified dynamics where we include additional particle moves that occur with an exponential rate
$\gamma >0$, i.e., we consider the piecewise deterministic Markov process with generator
\begin{equation}\label{eq:Lhatsplitting}
    \hatLgamma f \ =\ \hatL f\ +\ \gamma\,(Q f-f)\, ,\qquad\dom(\hatLgamma) = \dom(\hatL)\,,
\end{equation}
where $Q$ is the generator of a Markov chain on $\mathbb Z_n$ which only acts
on the second component of a function $f\colon\S\times\mathbb Z_n\to\mathbb R$. 
We denote the associated transition semigroup acting on $L^2(\hat\mu)$ by $(\hat P_t^{(\gamma)})_{t\geq 0}$.
Since $Q$ only acts on the second component, as in \Cref{thm:SRMlift}, this modified dynamics is still a lift of the process $(Z(t))_{t\geq 0}$ solving the discrete stochastic heat equation \eqref{eq:SHE}. In particular, the lower bound on the relaxation time in \Cref{thm:lowerbound} still applies to the modified dynamics.
If $Q=\Pi $, where
\begin{equation*}
    \Pi f(\ell,x)\ =\ \frac{1}{n}\sum_{i=1}^nf(\ell,i)\,,
\end{equation*}
then the second component is resampled from the uniform distribution on $\Z_n$ after independent $\mathrm{Exp} (\gamma )$-distributed waiting times. In this case, we can derive the following bound on convergence to equilibrium.

\begin{theorem}\label{thm:convrate}
    If $Q=\Pi$, then the transition semigroup $(\hat P_t^{(\gamma)})_{t\geq 0}$ satisfies
    \begin{equation}\label{eq:SRMconvergence}
        \norm{\hat P_t^{(\gamma)}f}_{L^2(\hat\mu)} \leq e^{-\nu (t-T)}\norm{f}_{L^2(\hat\mu)}\qquad\text{for all }f\in L_0^2(\hat\mu)\text{ and }t\geq 0
    \end{equation}
    with $T\in O(n)$ and inverse convergence rate
    $$\nu^{-1}\in O\left(\gamma n^3+\gamma^{-1}n\right).$$ 
    In particular,
    for $\gamma\propto n^{-1}$, the $L^2$-relaxation time is of order $O(n^{2})$.
\end{theorem}

We defer the proof of \Cref{thm:convrate} to \Cref{ssec:upperboundproof}. The point is that not only the process $(X^{(\gamma )}(t))$ but the joint process $(\mathcal L^{(\gamma )}(t),X^{(\gamma )}(t))$ converges to equilibrium with rate $\nu$.\medskip

In the physics literature, it has been conjectured  that even for $\gamma =0$, the $L^2$-relaxation time of the self-repellent random walk is of order $O(n^2)$, see the remark above. Unfortunately, the upper bound in \Cref{thm:convrate} degenerates in the limit $\gamma\to 0$.\medskip 

A dynamics that is somehow closer to the self-repellent random walk is obtained 
if $Q$ is the transition matrix of a simple random walk on $\Z_n$,
i.e.\ $Q(i,i\pm 1)=1/2$ for all $i\in\Z_n$.
\begin{theorem}\label{thm:convrate_rwrefresh}
    If $Q$ is the transition matrix of a simple random walk, then the corresponding transition semigroup $(\hat P_t^{(\gamma)})_{t\geq 0}$ satisfies
    \begin{equation*}
        \norm{\hat P_t^{(\gamma)}f}_{L^2(\hat\mu)} \leq e^{-\nu (t-T)}\norm{f}_{L^2(\hat\mu)}\qquad\text{for all }f\in L_0^2(\hat\mu)\text{ and }t\geq 0
    \end{equation*}
    with $T\in O(n^{3/2})$ and inverse convergence rate
    $$\nu^{-1}\in O\left(\gamma n^3+\gamma^{-1}n^2\right).$$ 
    In particular,
    for $\gamma\propto n^{-3/2}$, the $L^2$-relaxation time is of order $O(n^{5/2})$.
\end{theorem}

\Cref{thm:convrate_rwrefresh} is proved in a similar way as \Cref{thm:convrate}, taking into account the spectral gap of $Q$, see the proof given in \Cref{ssec:upperboundproof}.

\begin{remark}[Self-repellent motion] 
    The \emph{true self-repellent motion} was introduced by Tóth and Werner \cite{TothWerner1998TSRM} as a locally self-repelling stochastic process in continuous space and time, providing a continuous space counterpart to self-repellent random walks \cite{AmitParisiPeliti1983Asymptotic,Toth1995TrueSelfAvoiding,TothVeto2011Continuous}. Heuristically, it can be described as a process $(X(t))_{t\geq0}$ with drift given by the negative spacial derivative of its own local time density $L(t)$ at $X(t)$, even though the local time profile $x\mapsto L(t,x)$ is not differentiable. Precise invariance principles, i.e.\ convergence in distribution of rescaled self-repellent random walks to the true self-repellent motion, have been shown in \cite{newman2006convergence,kosygina2025convergence}. In the following, we give an informal presentation of how the self-repellent motion arises as a formal scaling limit of the process $(\mathcal L(t),X(t))_{t\geq 0}$ considered above, and show consequences for the relaxation time of the self-repellent motion on the torus.
    
    Consider the above dynamics on a discretisation of the one-dimensional torus $\mathbb{T}_N = \R/(N\Z)$ of length $N$ with spacing $h>0$, where we assume that $n=N/h\in\N$.  More precisely, consider the dynamics $(\mathcal L(t),X(t))_{t\geq 0}$ on $\S\times\Z_n$ with generator
    \begin{equation}\label{eq:SRMapproxGenerator}
        \hatL_h\ = \ \frac{1}{h}\T\ +\ \frac{1}{h^2}\J\,.
    \end{equation}
    The invariant probability measure is $\hat\mu_h = \mu_h\otimes\unif(\Z_n)$  with
    \begin{equation*}
        \mu_h(\diff\ell)\propto\exp(-U_h(\ell))\diff \ell\,, \qquad\text{where}\quad U_h(\ell) = \frac{h}{2}\sum_{i=1}^{N/h}\left(\frac{\ell_{i+1}-\ell_i}{h}\right)^2 = \frac{h}{2}\sum_{i=1}^n(\ell_i')^2
    \end{equation*}
    and $\ell_i' = \frac{1}{h}(\ell_{i+1}-\ell_i)$. By \Cref{lem:LhatProperties}, the symmetric and antisymmetric parts of $\hatL_h$ are
    \begin{align*}
        \frac{1}{2}(\hatL+\hatL^*)f(\ell,i) &= \frac{1}{2h}\Big(|\ell_{i-1}'|\big(f(\ell,i-1)-f(\ell,i)\big) +|\ell_i'|\big(f(\ell,i+1)-f(\ell,i)\big)\Big)\quad\text{and}\\
        \frac{1}{2}(\hatL-\hatL^*)f(\ell,i)&=\frac{1}{2h}\Bigl(\ell_{i-1}'\big(f(\ell,i-1)-f(\ell,i)\big)-\ell_i'\big(f(\ell,i+1)-f(\ell,i)\big)+\T f(\ell,i)\Big)\,.
    \end{align*}
    Viewing the dynamics $(\mathcal L(t),X(t))_{t\geq 0}$ on $\S\times\Z_n$ as a discretisation of a process on $C_0(\mathbb{T}_N)\times\mathbb{T}_N$, where $C_0(\mathbb{T}_N)$ denotes the space of continuous functions on $\mathbb{T}_N$ with mean zero, the above decomposition formally suggests that, in the limit $h\downarrow0$, the process should behave like the solution to the SDE
    \begin{align*}
        \diff \mathcal L_t &= \delta_{X_t}\diff t\\
        \diff X_t &= -\mathcal L_t'(X_t)\diff t + \sqrt{h\mathcal L_t'(X_t)}\diff B_t\,,
    \end{align*}
    where $(B_t)_{t\geq0}$ is a one-dimensional Brownian motion. In the formal continuum limit $h\downarrow0$, the diffusion coefficient vanishes, and the result is reminiscent of the true-self repellent motion.
    In this case, since $U_h(\ell)$ is a discretisation of $U(\ell) = \int_{\mathbb{T}_N}(\ell_t')^2\diff t$, the probability measures $\mu_h$ are expected to converge to the law of a zero-mean Brownian bridge on $\mathbb{T}_N$.
    More rigorously, in the limit $h\downarrow0$, the rescaled process
    \begin{equation*}
        (\tilde{\mathcal{L}}(t),\,\tilde X(t))_{t\geq 0}\ = \ (h^{1/2}\mathcal{L}(th^{-3/2}),\, hX(th^{-3/2}))_{t\geq 0}
    \end{equation*}
    converges in law to the self-repellent motion on $\mathbb{T}_N$ and its local time process, see \cite{TothWerner1998TSRM,kosygina2025convergence,newman2006convergence,Toth1995TrueSelfAvoiding}. The generator of the rescaled process arises from \eqref{eq:SRMapproxGenerator} as $h^{-3/2}(\T+\J) = h^{-3/2}\hatL$ via this change of variables. In particular, suppose that $(\hat P_t^{(h)})$ is the transition semigroup of the self-repellent random walk and its local time on $\S\times\Z_n$ with $n = N/h$. Then one would expect that 
    \begin{equation*}
        t_\rel^{\mathrm{SRM}}\ =\ \lim_{h\to 0}\,h^{3/2}\,t_\rel(\hat P^{(h)})\,
    \end{equation*}
    is the relaxation time of the self-repellent motion on $\mathbb{T}_N$ and its local time process, provided the limit exists. Therefore, a relaxation time of the order $t_\rel(\hat P^{(h)}) \in \Theta(n^{3/2})$ would yield $t_\rel^{\mathrm{SRM}}\in \Theta(N^{3/2})$. Such a bound on the relaxation time is consistent with the lower bound in \Cref{thm:lowerbound}, yet out of reach of the upper bound of order $O(n^2)$ in \Cref{thm:convrate}. While a relaxation time of the order $N^{3/2}$ for the self-repellent motion on $\mathbb{T}_N$ is plausible from the superdiffusive scaling $\mathbb{E}[|X_t|] = t^{2/3}\mathbb{E}[|X_1|]$ of the self-repellent motion on $\R$, we stress that a bound on the joint relaxation time of the self-repellent motion and its local time is a stronger statement.


\end{remark}

\section{Application to Markov Chain Monte Carlo methods}\label{sec:ECMC}

It has recently been noted in the physics literature that the self-repellent random walk can be seen as a special instance of an Event Chain Monte Carlo method \cite{Maggs2024Nonreversible,Maggs2025EventChain}. In Markov chain Monte Carlo (MCMC) methods, one constructs Markov processes with a given invariant probability measure to generate approximate samples from a target distribution. 
A standard approach is \emph{Hamiltonian Monte Carlo (HMC)} which combines Hamiltonian dynamics with velocity resamplings \cite{Bou-Rabee2017RHMC}. MCMC methods based on variants of HMC are widely used across many applications ranging from statistics \cite{CGH2017Stan} to molecular dynamics \cite{CLS2007MolecularDynamics}.
\emph{Event Chain Monte Carlo (ECMC)} is a more recently introduced class of MCMC methods based on non-reversible piecewise deterministic Markov processes that can be simulated in a rejection-free way by a thinning procedure \cite{Krauth2021ECMC}. It was originally introduced in the context of hard disks \cite{BernardKrauthWilson2009ECMC} where it enabled the discovery of a previously unknown phase transition.
Since then, it has been generalised to many other models \cite{MichelKapferKrauth2014ECMC,HarlandMichelKampmannKierfeld2017ECMC,MichelDurmusSenecal2020Forward,LeiKrauth2018Mixing}. The Markov processes underlying ECMC are characterised by moving one particle or coordinate at a time on a straight line, and ensuring invariance of the target probability measure by adding ``events'' at which the active particle or coordinate changes with some state-dependent rate. 

For MCMC methods, the rate of convergence of the law of the process at time $t$ towards its invariant distribution determines the efficiency of the scheme and is hence of crucial importance. Classical methods are mostly based on reversible Markov processes. Combined with local moves, reversibility limits the process to diffusive behaviour, prohibiting efficient exploration of the state space and hence impeding rapid convergence to stationarity. 
Non-reversible Markov processes such as ECMC possibly mitigate this issue. They can be constructed by lifting \cite{Chen1999Lift, Diaconis2000Lift, EberleLoerler2024Lifts} a reversible dynamics through introduction of additional variables and thereby enlarging the state space.
Mathematically, the analysis of these processes is challenging due to their non-reversibi\-lity and non-gradient nature, and rigorous results remain scarce or out of reach in situations of practical interest. For this reason, one is interested in studying
toy models that enable a benchmark comparison of different sampling schemes \cite{Krauth2024HMCvsECMC, Neal2004AsymptoticVariance,FBPR2018PDMP}. 

\subsection{Sampling chains of oscillators}

    One class of such toy models are chains of oscillators where $n$ particles are placed on a ring and each particle interacts only with its two neighbours. At inverse temperature $\beta\in (0,\infty )$, the corresponding Gibbs measure is
    $$\mu_\beta (\diff y)\propto\exp(-\beta U(y))\diff y,\quad\text{ where }\quad
        U(y) = \sum_{i=1}^nW\left(y_{i+1}-y_{i}\right)\, .$$
    Here $W\colon\R\to [0,\infty )$ is a continuously differentiable and symmetric interaction potential such that $\exp (-\beta W)$ is integrable.
    Without confinement, the Gibbs measure is not normalisable. This can be fixed by
    pinning the measure, for example by restricting to the subspace $\mathcal S\subset\R^n$ of configurations with mean zero.

    We consider a version of the Event Chain Monte Carlo algorithm which,
    when applied to a Gibbs measure as above, simulates a piecewise deterministic Markov process $(Y(t),X(t))_{t\geq 0}$ with state space $\R^n\times\mathbb Z_n$, where the position vector $Y(t)$ increases in the component corresponding to the active particle $X(t)$, i.e.,\ 
    \begin{align*}
        \frac{\diff}{\diff t}Y(t) \ =\ e_{X(t)}\,,
    \end{align*}
    and $X(t)$ is resampled from the uniform distribution on $\mathbb Z_n$ with rate $\gamma\in [0,\infty )$ and performs nearest-neighbour jumps
     $i \to i\pm1$ with rates $\beta\cdot\left(W'(y_{i\pm 1}-y_i)\right)_-$ where $y$ is the current value of $Y(t)$. The process $\mathcal Y (t)$ with invariant distribution given by the Gibbs measure on $\mathcal S$ is again obtained by projection, i.e.,
     $$\mathcal Y_i(t)\ =\ Y_i(t)-n^{-1}t\, .$$
    It can be verified similarly as in the proof of \Cref{lem:LhatProperties} that $(\mathcal Y(t),X(t))$ is a piecewise deterministic Markov process with state space
    $\mathcal S\times\mathbb Z_n$ and invariant measure $\mu_\beta\otimes \mathrm{Unif}(\mathbb Z_n)$.
   Defining the directional derivatives $\partial_{y_i}$ as in \eqref{projectedvectorfields}, the generator and its adjoint are given on functions $f\in C^1_b(\S\times\Z_n )$ by
\begin{eqnarray*}
     \hatL f (y,i)& =&   \partial_{y_i}f(y,i)\,+\, 
    \beta\,\left(W'(y_{i+ 1}-y_i)\right)_-\bigl(f(y,i+ 1)-f(y, i)\bigr)\\&&
    \,+\, 
    \beta\,\left(W'(y_i-y_{i- 1})\right)_+\bigl(f(y,i- 1)-f(y, i)\bigr)
     \,+\, 
    \frac \gamma n\sum_{j=1}^n\bigl(f(y,j)-f(y, i)\bigr)
    \, ,\\
    \hatL^* f (y,i)& =&   -\partial_{y_i}f(y,i)\,+\, 
    \beta\,\left(W'(y_{i+ 1}-y_i)\right)_+\bigl(f(y,i+ 1)-f(y, i)\bigr)\\&&
    \,+\, 
    \beta\,\left(W'(y_i-y_{i- 1})\right)_-\bigl(f(y,i- 1)-f(y, i)\bigr)
    \,+\, 
    \frac \gamma n\sum_{j=1}^n\bigl(f(y,j)-f(y, i)\bigr)
    \,,
\end{eqnarray*}

 In contrast, randomised Hamiltonian Monte Carlo methods \cite{Bou-Rabee2017RHMC,Neal2011HMC} with exponentially distributed integration times are related to piecewise deterministic Markov processes $(Q(t),V(t))$ where the deterministic motion is given by Hamiltonian dynamics, and the velocity component $V(t)$ is resampled from a normal distribution after
 independent waiting times that are exponentially distributed with a fixed parameter $\gamma\in (0,\infty )$. For the chain of oscillators, we consider randomised HMC based on the 
 Hamiltonian $H(q,v)=U(q)+\frac 12 |v|^2$.
The corresponding Hamiltonian dynamics 
    \begin{equation}\label{eq:HamDyn}
        \frac{\diff}{\diff t}q_i(t)= v_i(t),\quad\frac{\diff}{\diff t} v_i(t)=  \,W'\big(q_{i+1}(t)-q_i(t)\big)-W'\big(q_i(t)-q_{i-1}(t)\big)  \,,\quad i\in\Z_n,
    \end{equation}
    preserves the Boltzmann-Gibbs distribution 
    $$\hat\mu_\beta (\diff q\diff v)\ =\ \mu_\beta (\diff q)\,\mathcal N(0,\beta^{-1}\mathrm{I}_{\mathcal S})(\diff v)\ \propto\ \exp (-\beta H(q,v))\diff q\diff v\quad\text{on }\mathcal S\times\mathcal S.$$
    Therefore, the continuous time Markov process with generator
    \begin{eqnarray*}
      \mathfrak H f (q,v)& =& \sum_{i=1}^n \left( v_i\partial_{q_i}f(q,v)+\left(W'(q_{i+1}-q_i)-W'(q_i-q_{i-1})\right)\partial_{v_i}f(q,v)\right)\\
      &&\qquad\qquad +\gamma \int (f(q,w)-f(q,v))\,\mathcal N(0,\beta^{-1}\mathrm{I}_{\mathcal S})(\diff w) 
    \end{eqnarray*}
    has invariant measure $\hat\mu_\beta$.
    Different versions of randomised Hamiltonian Monte Carlo algorithms are obtained by replacing the
    Hamiltonian dynamics in \eqref{eq:HamDyn} by a time discretisation, usually based on a 
    symplectic integrator, and, possibly, including a Metropolis filter to adjust for the
    discretisation bias. 

    Note that since in Event Chain Monte Carlo, only one particle is moving at a time, the complexity for simulating the corresponding Markov process over a time unit can be assumed
    to be of order $\Theta(1)$. In contrast, for Hamiltonian Monte Carlo, the complexity  for simulating the process over one time unit is at least of order $n$. More precisely, it is known that for the standard discretisation based on the Verlet scheme, the discretisation step size has to be chosen of order $O(n^{-1/4})$ to adjust for the discretisation bias, resulting in an
    overall complexity of order $\Omega (n^{5/4})$ per time unit \cite{ApersGriblingSzilagyi2024HMC,Beskos2013OptimalTuningHMC}.

  \subsection{The harmonic chain: Event Chain vs.\ Hamiltonian Monte Carlo}

    In order to compare the relaxation times for the continuous time Markov processes corresponding to ECMC and HMC, we consider the case of a harmonic chain of oscillators where $W(y)=y^2$ and $\beta =1$. In this case, 
    the ECMC algorithm exactly recovers the local time process of the self-repellent random walk considered above. On the other hand, randomised HMC is a second-order lift of the discrete stochastic heat equation
    \begin{equation}\label{eq:SHE2}
        \diff Z(t)\ =\ \frac{1}{2}\Delta_nZ(t)\diff t\;+\; \diff B(t)
    \end{equation}
    with a Brownian motion $B$ on $\S$, see \cite[Example 3]{EberleLoerler2024Lifts}.
    Recall that the local time process of the self-repellent walk is a second-order lift 
    of \eqref{eq:SHE} which describes the same Markov process as \eqref{eq:SHE2} but slowed down by a factor $n$. Due to this factor, the spectral gap $m$ of the generator of \eqref{eq:SHE2} is of the order $\Theta(n^2)$, see \eqref{eq:gapcollapse}. Based on the second order lift property, it has been shown in
    \cite{Lu2022PDMP,EberleLoerler2024Lifts,EberleLoerler2024Spacetime} that the $L^2$ rate
    of convergence to stationarity of the Markov process with generator $\mathfrak H$ is of the order $\Theta\left(\min (\gamma , m\gamma^{-1} )\right)$. This result holds generally for convex potentials, and in the quadratic case it also follows by an explicit computation. In particular, for any choice of $\gamma$, the $L^2$ relaxation time is 
    at least of order $\Omega (m^{-1/2})=\Omega (n)$, and the optimal order is achieved for
    $\gamma \propto 1/n$.

     \begin{corollary}
        For the harmonic chain of $n$ particles, the computational complexity of the Event Chain Monte Carlo algorithm with particle refreshment at rate $\gamma\propto n^{-1}$ is of the order $O(n^2)$. On the same model, the computational complexity of randomised HMC is of the order $\Omega(n^{9/4})$ for any choice of $\gamma$.
    \end{corollary}
    \begin{proof}
        The complexity of the ECMC algorithm coincides with its relaxation time since the computational complexity per time unit if of order $\Theta(1)$, so that the complexity bound follows from the relaxation time bound in \Cref{thm:convrate}. Randomised HMC, on the other hand, has an overall complexity of order $\Omega(n^{5/4})$ per time unit, see the discussion above. The relaxation time of order $\Omega(n)$ thus yields a complexity of order $\Omega(n^{9/4})$.
    \end{proof}

For the harmonic chain, ECMC thus exhibits an efficiency gain compared to more standard MCMC methods such as HMC, and a substantial gain in comparison to traditional MCMC methods based on reversible Markov processes.
While similar gains have been observed empirically for chains of oscillators with Lennard-Jones interactions \cite{LeiKrauthMaggs2019ECMC} and hard sphere models \cite{LeiKrauth2018Mixing}, so that this behaviour is conjectured to hold more generally \cite{Krauth2024HMCvsECMC}, our result gives the first rigorous proof that ECMC methods can outperform HMC on specific models.



\section{Convergence to equilibrium}

Convergence to equilibrium of non-reversible Markov processes can be quantified using a \emph{flow Poincaré inequality}, which can be established for second-order lifts using the framework of \cite{EGHLM2025Convergence,EberleLoerler2024Spacetime}. For the modified dynamics with generator
\begin{equation*}
    \hatLgamma f \ =\ \hatL f\ +\ \gamma\,(Q f-f)\, ,\qquad\dom(\hatLgamma) = \dom(\hatL)
\end{equation*}
introduced in \eqref{eq:Lhatsplitting} above, the general framework simplifies and reduces to \Cref{thm:abstractconvergence}. Denote by $(\E,\dom(\E))$ the Dirichlet form associated to $(\L,\dom(\L))$, which is the extension of 
\begin{equation*}
    \E(f,g)=\langle f,-\L g\rangle_{L^2(\mu)}
\end{equation*}
to a closed symmetric bilinear form with domain $\dom(\E)$ given by the closure of $\dom(\L)$ with respect to the norm $\|f\|_{L^2(\mu)} + \E(f,f)^{1/2}$. Recall that the probability measure $\hat\mu = \mu\otimes\unif(\Z_n)$ is invariant for the Markov process generated by $\hatLgamma$.

\begin{theorem}\label{thm:abstractconvergence}
    Assume that
    \begin{enumerate}[(i)]
        \item the operator $(\hatL,\dom(\hatL))$ is a second-order lift of $(\L,\dom(\L))$ such that $g\circ\pi\in\dom(\hatL^*)$ and $\hatL^*(g\circ\pi) = -\hatL(g\circ\pi)$ for all $g\in\dom(\L)$;
        \item the operator $(\L,\dom(\L))$ has purely discrete spectrum on $L^2(\mu)$ and a spectral gap $m>0$, i.e.\ 
        \begin{equation*}
            \norm{g-\mu(g)}_{L^2(\mu)}^2\ \leq \ \frac{1}{m}\E(g)
        \end{equation*}
        for all $f\in\dom(\E)$;
        \item there exists a constant $C_1>0$ such that
        \begin{equation*}
            \big\langle \hatL(g\circ\pi),\hatL(f-\Pi f)\big\rangle_{L^2(\hat\mu)}\ \leq\ C_1\norm{\L g}_{L^2(\mu)}\norm{f-Qf}_{L^2(\hat\mu)}
        \end{equation*}
        for all $f\in\dom(\hatL)$ and $g\in\dom(\L)$;
        \item the transition matrix $Q$ has a spectral gap $m_Q$, i.e.\ 
        \begin{equation*}
            x^\top Qx\leq (1-m_Q)|x|^2
        \end{equation*}
        for all $x\in\R^n$ with $\sum_{i=1}^nx_i=0$.
    \end{enumerate}
    Then there exists a universal constant $C > 0$ such that, for any $T>0$ and $\gamma>0$, 
    the semigroup $(\hat P_t^{(\gamma)})$ generated by $(\hatLgamma,\dom(\hatLgamma))$ satisfies
    \begin{equation*}
        \lVert\hat P_t^{(\gamma)}f\lVert_{L^2(\hat\mu)}\ \leq\ e^{-\nu(t-T)}\lVert f\rVert_{L^2(\hat\mu)}\qquad\text{for all }t\geq 0\text{ and }f\in L_0^2(\hat\mu),
    \end{equation*}
    where the rate $\nu$ is given by 
    \begin{equation}\label{eq:nu}
         \nu = C\,\frac{\gamma}{\gamma^2/m+C_1^2 + \frac{1}{m_Q}(1 + \frac1{mT^2})}.
    \end{equation}
\end{theorem}
\begin{proof}
    The statement is a direct consequence of \cite[Theorem 12]{EGHLM2025Convergence}, simplified by the fact that we consider a refreshment of the active particle determined by the transition matrix~$Q$.
\end{proof}

Choosing $T=m^{-1/2}$ in \eqref{eq:nu} yields the inverse convergence rate
\begin{equation*}
    \nu^{-1}\in O\left(\gamma m^{-1} + \gamma^{-1}(m_Q^{-1}+C_1^2)\right)\,,
\end{equation*}
and optimising by choosing $\gamma \propto \sqrt{m(m_Q^{-1}+C_1^2})$ yields
\begin{equation*}
    \nu^{-1} \in O\Bigl(m^{-1/2}\bigl(m_Q^{-1/2}+C_1\bigr)\Bigr)\,.
\end{equation*}
If $Q=\Pi$, we can choose $m_Q=1$ and obtain $\nu^{-1}\in O(m^{-1/2}C_1)$ for the choice \mbox{$\gamma\propto m^{1/2}C_1$}.

\subsection{Upper bounds on relaxation times of self-repellent random walks}\label{ssec:upperboundproof}

We turn to the proof of the upper bounds on the $L^2$-relaxation time of the modified self-repellent random walks given in \Cref{thm:convrate} and \Cref{thm:convrate_rwrefresh}.

\begin{lemma}\label{lem:bochnerbound}
    We have
    \begin{equation*}
        \frac{1}{n^2}\int\norm{\nabla^2 g}_F^2\diff\mu + \frac{1}{n^2}\sum_{i=1}^n\int ((\partial_{\ell_i}-\partial_{\ell_{i-1}})g)^2 \diff\mu = \int (2\L g)^2 \diff\mu
    \end{equation*}
    where $\norm{\nabla^2g}_F^2$ denotes the squared Frobenius norm of $\nabla^2g$.
\end{lemma}
\begin{proof}
    Since $\mu(\diff\ell)\propto\exp(\frac{1}{2}\ell^\top\Delta_n\ell)\diff\ell$, where $\diff\ell$ denotes the Haar measure on $\S$, Bochner's identity yields
    \begin{equation*}
        \int (2\L g)^2\diff\mu = \frac{1}{n^2}\int(\nabla^*\nabla g)^2\diff\mu=  \frac{1}{n^2}\int \norm{\nabla^2g}_F^2\diff\mu + \frac{1}{n^2}\int\nabla g^\top\Delta_n\nabla g\diff\mu\,.
    \end{equation*}
    The claim then follows from the relation
    \begin{equation*}
        \sum_{i=1}^n ((\partial_{\ell_i}-\partial_{\ell_{i-1}})g)^2  =  \nabla g^\top\Delta_n\nabla g.\qedhere
    \end{equation*}
\end{proof}

\begin{lemma}\label{lem:ibpboundmu}
    For any $v,w\in\S$ and $g\in C_b^2(\S)$, we have
    \begin{equation*}
        \int (v^\top\ell)^2(\partial_wg)^2\,\mu(\diff\ell) \leq 2v^\top(-\Delta_n)^{-1}v\int(\partial_wg)^2\diff\mu + 4\int(\partial_{\Delta_n^{-1}v,w}^2g)^2\diff\mu 
    \end{equation*}
    where $\Delta_n$ is the discrete Laplacian given by \eqref{eq:Deltan}.
\end{lemma}
\begin{proof}
    An integration by parts with respect to the Gaussian measure $\mu$ gives
    \begin{equation*}
        -\int v^\top\nabla f\diff\mu = \int v^\top\Delta_n\ell\,f(\ell)\,\mu(\diff\ell)
    \end{equation*}
    for all $f\in C_b^1(\S)$. Applying this with $f(\ell) = v^\top\Delta_n\ell (\partial_wg)^2$ and the inequality $2ab\leq\frac{1}{2}a^2+2b^2$ yields
    \begin{align*}
        \MoveEqLeft\int (v^\top\Delta_n\ell)^2(\partial_wg)^2\,\mu(\diff\ell) = -\int v^\top\nabla(v^\top\Delta_n\ell (\partial_wg)^2)\,\mu(\diff\ell)\\
        &=v^\top(-\Delta_n)v\int(\partial_w^2g)\diff\mu - 2\int v^\top\Delta_n\ell\partial_{vw}^2g(\ell)\partial_wg(\ell)\,\mu(\diff\ell) \\
        &\leq v^\top(-\Delta_n)v\int(\partial_w^2g)\diff\mu + \frac{1}{2}\int (v^\top\Delta_n\ell)^2(\partial_wg)^2\,\mu(\diff\ell) + 2\int (\partial_{v,w}^2g)^2\diff\mu\,.
    \end{align*}
    Rearranging and replacing $v$ by $\Delta_n^{-1}v$ yields the claim.
\end{proof}

\begin{proof}[Proof of \Cref{thm:convrate}]
    The claim follows from \Cref{thm:abstractconvergence} if we can show that conditions (i)--(iv) are satisfied with $m \in\Theta(n^{-3})$, $C_1^2\in O(n)$, and $m_Q=1$.
    
    To this end, condition (i) is satisfied by \Cref{lem:LhatProperties} and \Cref{thm:SRMlift}. The operator $(\L,\dom(\L))$ has discrete spectrum on $L^2(\mu)$ and a spectral gap $m\in\Theta(n^{-3})$ by \eqref{eq:gapcollapse}, so that condition (ii) is satisfied. Condition (iv) is immediately satisfied with $m_Q=1$ in case $Q=\Pi$. It remains to verify condition (iii). By the Cauchy-Schwarz inequality, it suffices to bound
    \begin{equation*}
        \norm{\hatL^*\hatL(g\circ\pi)}_{L^2(\hat\mu)}\leq C_1\norm{\L g}_{L^2(\mu)}.
    \end{equation*}
    By \Cref{lem:LhatProperties},
    \begin{align*}
        \hatL^*\hatL (g\circ\pi)(\ell,x) = -\partial_{\ell_x}^2g(\ell) &- (\ell_x-\ell_{x-1})_-\left(\partial_{\ell_x}g(\ell)-\partial_{\ell_{x-1}} g(\ell)\right) \\
        &+ (\ell_{x+1}-\ell_{x})_+\left(\partial_{\ell_{x+1}}g(\ell)-\partial_{\ell_{x}} g(\ell)\right).
    \end{align*}
    Hence
    \begin{align*}
        \MoveEqLeft\norm{\hatL^*\hatL(g\circ\pi)}_{L^2(\hat\mu)}^2 = \frac{1}{n}\sum_{i=1}^n\int\big(\big(\hatL^*\hatL(g\circ\pi)\big)(\ell,i)\big)^2\mu(\diff\ell)\\
        &\leq \frac{3}{n}\sum_{i=1}^n\int(\partial_{\ell_i}^2g(\ell))^2\mu(\diff \ell) + \frac{3}{n}\sum_{i=1}^n\int(\ell_i-\ell_{i-1})^2\left(\partial_{\ell_i}g(\ell)-\partial_{\ell_{i-1}} g(\ell)\right)^2\mu(\diff\ell).
    \end{align*}
    For the second summand, we apply \Cref{lem:ibpboundmu} with $v_i=w_i=e_i-e_{i-1}$ to obtain
    \begin{align*}
        \MoveEqLeft\frac{3}{n}\sum_{i=1}^n\int(\ell_i-\ell_{i-1})^2\left(\partial_{\ell_i}g(\ell)-\partial_{\ell_{i-1}} g(\ell)\right)^2\mu(\diff\ell)\\
        &\leq \frac{6}{n}\sum_{i=1}^nv_i^\top(-\Delta_n)^{-1}v_i\int((\partial_{\ell_i}-\partial_{\ell_{i-1}})g)^2\diff\mu + \frac{12}{n}\sum_{i=1}^n\int\big(\partial_{\Delta_n^{-1}v_i,v_i}^2g\big)^2\diff\mu\,.
    \end{align*}
    Note that $u_i = \Delta_n^{-1}v_i$ is given by
    \begin{equation*}
        (u_i)_j = \left(1-\frac{1}{n}\right)\cdot\begin{cases}
            \frac{1}{2}&\text{if }j=i-1,\\
            -\frac{1}{2}&\text{if }j = i,\\
            -\frac{1}{2}+\frac{k}{n-1}&\text{if }j = i+k,\,k\geq0\,.
        \end{cases}
    \end{equation*}
    In particular, $u_{i+1}-u_i = e_i-\frac{1}{n}\mathbf{1}=\tilde e_i$ is the projection of the $i$-th basis vector in $\R^n$ onto $\S$,
    \begin{equation*}
        v_i^\top (-\Delta_n)^{-1}v_i = -v_i^\top u_i = (u_i)_i-(u_i)_{i-1} = 1-\frac{1}{n}\,,
    \end{equation*}
    and, since $u_i=\sum_{j=1}^n(u_i)_j e_j=\sum_{j=1}^n(u_i)_j\tilde e_j$ and $v_i = \tilde e_i-\tilde e_{i-1}$,
    \begin{equation*}
        \partial_{u_i,v_i}^2 g= \sum_{j=1}^n(u_i)_j\partial_{\tilde e_i,\tilde e_j}^2g-\sum_{j=1}^n(u_i)_j\partial_{\tilde e_{i-1},\tilde e_j}^2g = \sum_{j=1}^n(u_i-u_{i+1})_j\partial_{\tilde e_i,\tilde e_j}^2g = -\partial_{\tilde e_i,\tilde e_i}^2g.
    \end{equation*}
    In the final equality we used that $u_i-u_{i+1} = -\tilde e_i$ and $\sum_{j=1}^n(\tilde e_i)_j\tilde e_j = \sum_{j=1}^n(\tilde e_i)_je_j = \tilde e_i$.
    Therefore,
    \begin{align*}
        \norm{\hatL^*\hatL(g\circ\pi)}_{L^2(\hat\mu)}^2&\leq \frac{6}{n} \sum_{i=1}^n\int((\partial_{\ell_i}-\partial_{\ell_{i-1}})g)^2\diff\mu +\frac{15}{n}\sum_{i=1}^n\int(\partial_{\ell_i}^2g)^2\diff\mu\\
        &\leq 60n\int(\L g)^2\diff\mu
    \end{align*}
    by \Cref{lem:bochnerbound}, so that assumption (iii) of \Cref{thm:abstractconvergence} is satisfied with $C_1 = \sqrt{60n}$.
\end{proof}

\begin{proof}[Proof of \Cref{thm:convrate_rwrefresh}]
    As in the case of \Cref{thm:convrate}, the claim follows from \Cref{thm:abstractconvergence} with $m_Q \in O(n^2)$.
\end{proof}

\section*{Acknowledgments}

The authors were funded by the Deutsche Forschungsgemeinschaft (DFG, German Research Foundation) under Germany’s Excellence Strategy EXC 2047 -- 390685813 as well as under CRC 1720 -- 539309657. We would like to thank Werner Krauth, Brune Massoulié and Stefan Oberdörster for helpful discussions.

\printbibliography

\end{document}